\documentclass[12pt,english]{amsart}

\input xy
\xyoption{all}

\usepackage{comment}
\usepackage[latin1]{inputenc}
\usepackage{amsfonts,amsmath,amssymb,amsthm}
\usepackage{stmaryrd} 
\usepackage{footmisc}

\newcommand{\Cc}{\mathbb{C}}

\newcommand{\Qq}{\mathbb{Q}}

{\theoremstyle{plain}
\newtheorem{theorem}{Theorem}[section]    

\newtheorem{twisting lemma}[theorem]{Twisting lemma}
\newtheorem{lemma}[theorem]{Lemma}       
\newtheorem{proposition}[theorem]{Proposition}  

}
{\theoremstyle{remark}

}

\def\cm{\hbox{\hbox{\rm C}\kern-5pt{\raise 1pt\hbox{$|$}}}}

\def\lhfl#1#2{\smash{\mathop{\hbox to 12mm{\leftarrowfill}}
\limits^{#1}_{#2}}}
\def\rhfl#1#2{\smash{\mathop{\hbox to 12mm{\rightarrowfill}}
\limits^{#1}_{#2}}}

\def\build#1_#2^#3{\mathrel{
\mathop{\kern 0pt#1}\limits_{#2}^{#3}}}

\def\htrait#1#2{\smash{\mathop{\hbox to 12mm{\hrulefill}}
\limits^{#1}_{#2}}}

\def\sxbullet{{\raise 2pt\hbox{\bf .}}}
\numberwithin{equation}{section}

\begin{document}

\title{Automorphism groups over Hilbertian fields}

\author{Fran\c cois Legrand}

\email{legrandfranc@technion.ac.il}

\address{Department of Mathematics, Technion - Israel Institute of Technology, Haifa 32000, Israel}

\author{Elad Paran}

\email{paran@openu.ac.il}

\address{Department of Mathematics and Computer Science, the Open University of Israel, Ra'anana 4353701, Israel}

\date{\today}

\maketitle

\vspace{-0.5mm}

\begin{abstract}
We show that every finite group occurs as the automorphism group of infinitely many finite (field) extensions of any given Hilbertian field. This extends and unifies previous results of M. Fried and Takahashi on the global field case.
\end{abstract}

\vspace{-0.5mm}

\section{Introduction}

In the present paper, we are interested in the following rudimentary question in field theory: {\it{given an arbitrary field $k$, is every finite group (isomorphic to) the automorphism group of a finite (field) extension $L/k$?}}\footnote{This question is weaker than the so-called {\it{Inverse Galois Problem}} as we do not require the extension $L/k$ to be Galois.} First negative conclusions on this natural problem can be derived from any classical book in Galois theory. For example, if $k=\Cc$, then, the answer is obviously negative as $\Cc$ has only one finite extension (namely, the trivial one $\Cc/\Cc$). Other similar examples are the field of real numbers, finite fields or the field $\kappa((T))$ of Laurent series with coefficients in any algebraically closed field $\kappa$ of characteristic zero.

In \cite{FK78}, E. Fried and J. Koll\'ar gave a positive answer in the case $k=\Qq$. However, M. Fried found a gap in their proof and gave a complete and different argument in \cite{Fri80}, even proving that there exist infinitely many finite extensions of $\Qq$ with a specified automorphism group by using {\it{Hilbert's irreducibility theorem}} \cite{Hil92}. More generally, the proof given in \cite{Fri80} works {\it{verbatim}} for every {\it{Hilbertian}}\footnote{Recall that a field $k$ is {\it{Hilbertian}} if Hilbert's irreducibility theorem holds for $k$ instead of $\Qq$. See, e.g., \cite{FJ08} for more on Hilbertian fields.} field of characteristic zero. Independently, Takahashi \cite{Tak80} proved that the answer is positive for all global fields, providing in particular the first examples in positive characteristic. Finally, in \cite{Gey83}, Geyer offered an elementary proof of M. Fried's result (for $\Qq$) which does not use Hilbert's irreducibility theorem.

A classical conjecture in field arithmetic asserts that every finite group occurs as the automorphism group of some finite {\it{Galois}} extension of any given Hilbertian field; see, e.g., \cite[\S2.1.1]{DD97b}. It is then natural to ask whether the previous results of M. Fried and Takahashi hold for arbitrary Hilbertian fields. We show that the answer is positive, thus extending and unifying all the previous particular cases mentioned above:

\begin{theorem} \label{thm}
Every finite group $G$ occurs as the automorphism group of infinitely many distinct finite extensions of any given Hilbertian field $k$.
\end{theorem}

{\bf{Acknowledgments.}} We wish to thank Dan Haran for helpful comments and the anonymous referee for suggesting a more elementary proof of Proposition \ref{lemma 4}. The first author is partially supported by the Israel Science Foundation (grants No. 693/13 and 577/15). The second author is partially supported by the Israel Science Foundation (grant No. 696/13).

\vspace{3mm}

\noindent
{\it{Notation.}} Throughout the paper, let $k$ be a field, $T$ an indeterminate over $k$, $\Omega$ an algebraic closure of $k(T)$, $\overline{k}$ the algebraic closure of $k$ inside $\Omega$, and $G$ a finite group.

\section{Preliminaries}

\subsection{Field theoretic background}

Given a finite (field) extension $L/k$ (with $L \subseteq \overline{k}$), recall that the {\it{automorphism group of $L/k$}}, denoted by ${\rm{Aut}} (L/k)$, is the group of all (field) automorphisms of $L$ which fix every element of $k$. This group is finite and one has $|{\rm{Aut}} (L/k)| \leq [L:k]$. The extension $L/k$ is {\it{Galois}} (that is, normal and separable) if and only if $|{\rm{Aut}} (L/k)| = [L:k]$. In this case, the automorphism group of $L/k$ is the {\it{Galois group of $L/k$}} and is denoted by ${\rm{Gal}}(L/k)$.

Recall the following classical lemma:

\begin{lemma} \label{lemma 5}
Assume that the extension $L/k$ is Galois. Given an intermediate field $k \subseteq F \subseteq L$, denote the normalizer of ${\rm{Gal}}(L/F)$ in ${\rm{Gal}}(L/k)$ by $N$. Then, the groups ${\rm{Aut}}(F/k)$ and $N/{\rm{Gal}}(L/F)$ are isomorphic.
\end{lemma}

\subsection{Function field extensions}

Let $E/k(T)$ be a finite separable extension with $E \subseteq \Omega$, $E \overline{k}$ the compositum of $E$ and $\overline{k}(T)$ inside $\Omega$, and $\widehat{E}/k(T)$ the Galois closure of $E/k(T)$ inside $\Omega$. Let $P(T,X) \in k[T][X]$ be the minimal polynomial over $k(T)$ of a primitive element $x(T)$ of $\widehat{E}/k(T)$, assumed to be integral over $k[T]$. Moreover,  let $a(T,X) \in k[T][X]$ be such that $a(T,x(T))$ is a primitive element of $E/k(T)$. Denote the finitely many elements $t \in \overline{k}$ such that $P(t,X)$ has a multiple root in $\overline{k}$ by $t_1,\dots,t_r$. Given $t_0 \in k \setminus \{t_1,\dots,t_r\}$, the splitting field over $k$ of $P(t_0,X)$ inside $\overline{k}$ is denoted by $\widehat{E}_{t_0}$. Moreover, if $P(t_0,X)$ is irreducible over $k$, pick a root $x(t_0) \in \overline{k}$ of $P(t_0,X)$ and denote the field $k(a(t_0,x(t_0))$ by $E_{t_0}$.

\begin{proposition} \label{prop 1}
Assume that $k$ is Hilbertian and $G$ is isomorphic to the automorphism group of $E/k(T)$. Then, there exist infinitely many $t_0 \in k$ such that the automorphism group of the (separable) extension $E_{t_0}/k$ is isomorphic to $G$. Moreover, infinitely many of these extensions $E_{t_0}/k$ may be chosen to be distinct if the following condition holds:
\begin{equation} \label{eq}
[E\overline{k}:\overline{k}(T)] \geq 2.
\end{equation}
\end{proposition}

\begin{proof}
As $k$ is Hilbertian, there exist infinitely many $t_0 \in k$ such that $P(t_0,X)$ is irreducible over $k$. In particular, for such a $t_0$ (up to finitely many), there exists an isomorphism
$$\psi_{t_0}:{\rm{Gal}}(\widehat{E}_{t_0}/k) \rightarrow {\rm{Gal}}(\widehat{E}/k(T))$$
that satisfies $\psi_{t_0}({\rm{Gal}}(\widehat{E}_{t_0} / {E}_{t_0}))={\rm{Gal}}(\widehat{E}/E)$ \cite[Lemma 16.1.1]{FJ08} \cite[\S1.9]{Deb09}. Denote the normalizer of ${\rm{Gal}}(\widehat{E}/E)$ in ${\rm{Gal}}(\widehat{E}/k(T))$ by $N_T$ and the one of ${\rm{Gal}}( \widehat{E}_{t_0}/E_{t_0})$ in ${\rm{Gal}}(\widehat{E}_{t_0}/k)$ by $N_{t_0}$. Then, $\psi_{t_0}$ induces an isomorphism between $N_T/{\rm{Gal}}(\widehat{E}/E)$ and $N_{t_0} / {\rm{Gal}}(\widehat{E}_{t_0} / E_{t_0})$. Moreover, by Lemma \ref{lemma 5}, one has $N_T/{\rm{Gal}}(\widehat{E}/E) \cong {\rm{Aut}}(E/k(T))$ and $N_{t_0}/{\rm{Gal}}(\widehat{E}_{t_0} / {E}_{t_0}) \cong {\rm{Aut}}(E_{t_0}/k).$ Hence, ${\rm{Aut}}(E_{t_0}/k) \cong G$. The more precise conclusion under condition \eqref{eq} is quite standard and details are then left to the interested reader.
\end{proof}

\subsection{A function field theoretic result}

Proposition \ref{lemma 4} below is a key tool in the proof of Theorem \ref{thm}, which is given in \S3.

\begin{proposition} \label{lemma 4}
Given an intermediate field $k \subseteq K \subseteq \overline{k}$ and $y \in K$, set $$P_y(T,X):=X^3 + (T-y) X + (T-y) \in K[T][X].$$ 

\noindent
{\rm{(1)}} The polynomial $P_y(T,X)$ is irreducible over $K(T)$ and separable. Moreover, it has Galois group isomorphic to $S_3$ over $K(T)$.

\vspace{0.5mm}

\noindent
{\rm{(2)}} Denote by $K_y$ the field generated over $K(T)$ by any given root of $P_y(T,X)$ inside $\Omega$. Then, given $y_1 \not=y_2$ in $K$, one has $K_{y_1} \not= K_{y_2}$.
\end{proposition}

\begin{proof}
For (1), we refer, e.g., to \cite[\S2.1]{JLY02}. Now, we prove (2). Let $x_1$ (resp., $x_2$) be a root in $\Omega$ of $P_{y_1}(T,X)$ (resp., of $P_{y_2}(T,X)$). We show below that the fields $K(T)(x_1)$ and $K(T)(x_2)$ are distinct. Using that $$(1+x_i)y_i = T(1+x_i)+{x_i^3}$$ for each $i \in \{1,2\}$ and setting $$\delta:=y_2-y_1,$$ we see that $x_2$ is a root of the polynomial
$$F(X):=(x_1+1)X^3 - ((x_1 + 1)\delta + x_1^3)(X+1) \in K(x_1)[X].$$
By (1), $x_1$ is transcendental over $K$. Then, consider the following polynomial:
$$G(Y):= (Y + 1)\delta + Y^3 \in K[Y].$$
Since $\delta \not=0$, the polynomial $G(Y)$ has at least one root in $\overline{k}$, say $\mu$, of multiplicity 1. Then, $F(X)$ is an Eisenstein polynomial with respect to the prime element $x_1-\mu$ of $K[x_1]$, hence irreducible over $K(x_1)$. In particular, $x_2$ is not in $K(x_1)$. But one has $K(x_1)=K(T)(x_1)$ since $$T = y_1 - \frac{x_1^3}{1+x_1}.$$
Hence, (2) holds.
\end{proof}

\section{Proof of Theorem \ref{thm}}

By Proposition \ref{prop 1}, it suffices to prove Proposition \ref{prop 2} below to get the conclusion of Theorem \ref{thm}.

\begin{proposition} \label{prop 2}
Assume that $k$ is Hilbertian. Then, there exists a finite separable extension of $k(T)$ (contained in $\Omega$) whose automorphism group is isomorphic to $G$ and which satisfies condition \eqref{eq} of Proposition \ref{prop 1}.
\end{proposition}

\begin{proof}

Let $n \geq 1$ be an integer such that $G$ is isomorphic to a subgroup of $S_n$. We may and will assume that $G$ itself is a subgroup of $S_n$.

First, let $L/k$ be a finite Galois extension with $L \subseteq \overline{k}$ and whose Galois group is isomorphic to $S_n$; such an extension exists as $k$ has been assumed to be Hilbertian\footnote{Indeed, given algebraically independent indeterminates $T_1,\dots,T_n$ over $k$, recall that $S_n$ acts on $k(T_1,\dots,T_n)$ by permuting the variables. The fixed field is the field of symmetric functions in these variables. By the fundamental theorem of symmetric functions, this field is generated over $k$ by the elementary symmetric functions $U_1,\dots,U_n$, which are algebraically independent over $k$. Thus, the fixed field $k(U_1,\dots,U_n)$ is a rational function field. By the Hilbertianity of $k$, $S_n$ can be realized over $k$. See, e.g., \cite[Chapter VI, \S2, Example 4]{Lan02} for more details.}. Fix an isomorphism $\phi : {\rm{Gal}}(L/k) \rightarrow S_n$ and set $G':=\phi^{-1}(G).$ As the finite extension $L/k$ is separable, the same holds for the subextension $L^{G'}/k$, where $L^{G'}$ is the fixed field of $G'$ in $L$. Let $y \in L^{G'}$ be such that $L^{G'} = k(y).$ By Proposition \ref{lemma 4}, $$P_y(T,X)=X^3 + (T-y) X + (T-y) \in L^{G'}[T][X]$$ is irreducible over $L^{G'}(T)$. Let $x(T) \in \Omega$ be a root and $E$ the compositum of $L(T)$ and $L^{G'}(T, x(T))$ inside $\Omega$. It is clear that $E=L(T,x(T)).$ 

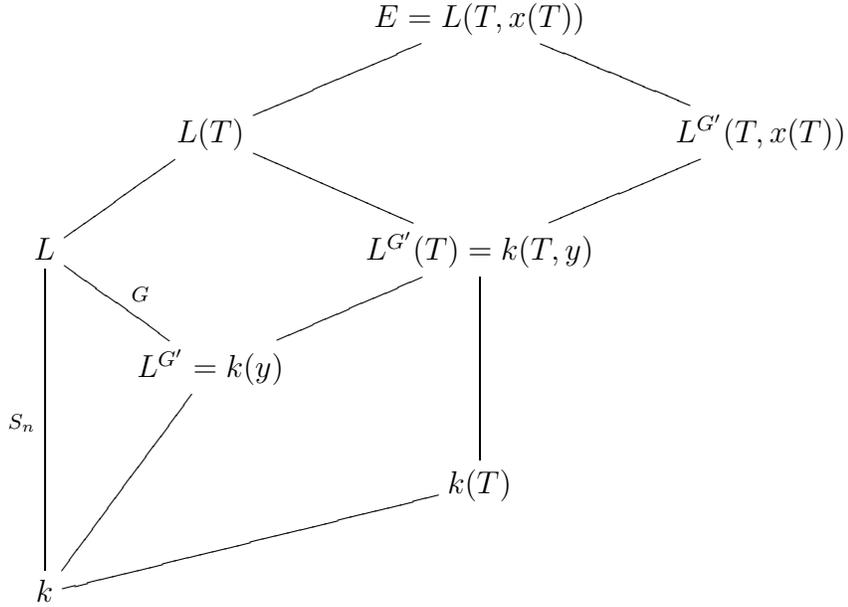
\begin{figure}[h!]
\[ \xymatrix{
& & E=L(T,x(T)) \ar@{-}[rd] &  \\
& L(T) \ar@{-}[ur] & & L^{G'}(T,x(T)) \\
L \ar@{-}[ur] & & L^{G'}(T)=k(T,y) \ar@{-}[lu] \ar@{-}[ru]&& \\
& L^{G'}=k(y) \ar@{-}[ru] \ar@{-}[lu]_{G} &    &\\
& & k(T) \ar@{-}[uu] & \\
k \ar@{-}[rur] \ar@{-}[uuu]^{S_n} \ar@{-}[ruu]& & & \\
}
\]
\caption{Field extensions}\label{Fig1a}
\end{figure}

Now, we compute the automorphism group of $E/L^{G'}(T)$:

\begin{lemma} \label{lemma 1}
One has ${\rm{Aut}}(E/L^{G'}(T)) \cong G$.
\end{lemma}

\begin{proof}
By Proposition \ref{lemma 4}, $P_y(T,X)$ is irreducible over $L(T)$, that is, the fields $L(T)$ and $L^{G'}(T,x(T))$ are linearly disjoint over $L^{G'}(T)$. As the finite extension $L(T)/L^{G'}(T)$ is Galois and has Galois group isomorphic to $G$, the same is true for $E/L^{G'}(T,x(T)).$ It then suffices to show that any given automorphism $\sigma$ of $E/L^{G'}(T)$ fixes $x(T)$. Assume that $\sigma$ does not. Then, $\sigma(x(T))$ is another root of $P_y(T,X)$ and it is in $E$. Hence, $E$ contains all the roots of $P_y(T,X)$ (as $P_y(T,X)$ has degree 3 in $X$). By Proposition \ref{lemma 4}, we get $[E:L(T)] \geq 6$, a contradiction.
\end{proof}

Next, we determine the automorphism group of $E/k(T)$:

\begin{lemma} \label{lemma 2}
One has ${\rm{Aut}}(E/k(T)) \cong G$.
\end{lemma}

\begin{proof}
By Lemma \ref{lemma 1}, it suffices to prove that $${\rm{Aut}}(E/L^{G'}(T))={\rm{Aut}}(E/k(T)).$$ Clearly, the former is a subgroup of the latter. For the converse, let $\sigma$ be an element of ${\rm{Aut}}(E/k(T))$. Suppose that $\sigma$ is not in ${\rm{Aut}}(E/L^{G'}(T))$. Then, one has $\sigma(y) \not=y$ and $\sigma(x(T))$ is a root of 
$$P_{\sigma(y)}(T,X)=X^3 + (T-\sigma(y))X + (T-\sigma(y)))  \in L[T][X].$$
Proposition \ref{lemma 4} then gives $L(T, \sigma(x(T))) \not=L(T,x(T))(=E)$. As $\sigma$ is in ${\rm{Aut}}(E/k(T))$, we get that the left-hand side field is strictly contained in $E$, which cannot happen as both fields have degree 3 over $L(T)$.
\end{proof}

Finally, it remains to notice that $[E\overline{k}:\overline{k}(T)]=3$ (Proposition \ref{lemma 4}) as needed for condition \eqref{eq} of Proposition \ref{prop 1}. Hence, the conclusion of Proposition \ref{prop 2} holds.
\end{proof}

\bibliography{Biblio2}
\bibliographystyle{alpha}

\end{document}